\documentclass{amsart}
\usepackage{amssymb}
\usepackage{amsmath}
\usepackage{latexsym}

\usepackage{amsthm}

\newtheorem{prop}{Proposition}[section]
\newtheorem{thm}[prop]{Theorem}

\newtheorem{cor}[prop]{Corollary}

\theoremstyle{definition}

\newtheorem{ex}[prop]{Example}

\begin{document}
\title{On a Class of Half-Factorial Domains}

\author{Mark Batell}
\address{Department of Mathematical Sciences\\
	University of Cincinnati\\
	Cincinnati, OH 45221}
\email{mark.batell@uc.edu}

\keywords{Factorization, unique factorization}
\subjclass[2000]{Primary: 13F15}

\begin{abstract}
Let $R$ be an integral domain. For elements $a,b \in R$, let $[a,b]$ denote their greatest common divisor, if it exists. We say that $R$ has the \emph{Z-property} if whenever $a,b,c,d$ and $e$ are nonzero nonunits of $R$ such that $abc=de$, then $[ab,d] \neq 1$ or $[ab,e] \neq 1$. The purpose of this paper is to study this property. The atomic integral domains that have this property constitute a class of half-factorial domains. Also, it is known that $R$ must have this property in order for the polynomial ring $R[x]$ to be half-factorial.  We use it to give a characterization of half-factorial polynomial rings in the case where every $v$-ideal is $v$-generated by two elements. We also show that if $R$ is a Krull domain with this property, then $R$ has torsion class group. 
\end{abstract}

\maketitle

\section{Introduction}

Let $R$ be an integral domain with quotient field $K$. For a pair of elements $a,b \in R$ we shall use the symbol $[a,b]$ to denote their greatest common divisor, if it exists. If $x$ is an indeterminate and $f \in K[x]$ is a polynomial, the symbol $A_f$ will be used to denote the fractional ideal generated by the coefficients of $f$. A \emph{primitive ideal} is an ideal $I \subseteq R$ such that $I=A_f$ for some primitive polynomial $f \in R[x]$.  Note that a polynomial of degree one, say $f=a_1x+a_0$, is irreducible in $R[x]$ if and only if it is primitive.

An integral domain $R$ is a \emph{half-factorial domain} (HFD) if $R$ is atomic and $m=n$ whenever $\alpha_1 \alpha_2 \cdots \alpha_n=\beta_1 \beta_2 \cdots \beta_m$ and the $\alpha$'s and $\beta$'s are irreducible elements (atoms) of $R$. We shall say that $R$ has the \emph{Z-property} if whenever $a,b,c,d,e \in R$ are nonzero nonunits such that $abc=de$, then $[ab,d] \neq 1$ or $[ab,e] \neq 1$, i.e., $ab$ and $d$ have a common factor or $ab$ and $e$ have a common factor. Atomic domains that satisfy the Z-property form a class of half-factorial domains (see Theorem \ref{z} below).

\begin{ex}
We show that if $R$ is the domain given by

\[
R:=F[x,y,zx,zy] 
\]\

\noindent then $R$ does not have the Z-property. Here, $F$ is a field with two elements and $x,y,z$ are indeterminates. In $R$ we have the equation 

\[
(x)(x)(zy)(zy)=(zx)(zx)(y)(y)
\]\

\noindent and it is easy to see that $[x^2,y^2]=1$ and $[x^2,z^2x^2]=1$. It follows that $R$ does not have Z-property.
\end{ex}

If $R$ is an HFD, then by definition any two irreducible factorizations of a nonzero nonunit have the same length. For a nonzero nonunit $x$ satisfying this condition, we define the length $\ell(x)$ to be the number of irreducible factors that appear in a factorization of $x$. 

In this paper, we consider the half-factorial property in polynomial rings and the closely related Z-property. We will establish a new result for polynomial rings in addition to some links between the Z-property and the half-factorial property. In order to provide the proper context, we will now give the statements and proofs of some earlier results from \cite{B}. 

\begin{thm} \cite[Proposition 3.1]{B}
Let $R$ be a domain and $x$ an indeterminate. If $R[x]$ is an HFD, then $R$ has the Z-property.
\end{thm}

\begin{proof}
Assume $a,b,c,d,e \in R$ are nonzero nonunits such that $abc=de$, $[ab,d]=1$ and $[ab,e]=1$.  We consider the polynomials $f=abx+d$ and $g=abx+e$ in $R[x]$. Note that $f$ and $g$ are irreducible in $R[x]$ and that

\[
fg=abh
\]\

\noindent where $h=abx^2+(d+e)x+c$. It follows that $R[x]$ is not an HFD. 
\end{proof}

\begin{thm} \label{z} \cite[Proposition 3.2]{B}
Let $R$ be an atomic domain. If $R$ has the Z-property, then $R$ is an HFD. 
\end{thm}

\begin{proof}
Assume 
\begin{equation}\label{original}
\alpha_1\alpha_2 \cdots \alpha_n=\beta_1\beta_2 \cdots \beta_m 
\end{equation} 
where the $\alpha$'s and $\beta$'s are irreducible elements of $R$.  We assume $n>m$ and derive a contradiction. Consider first the case $m=2$. Then $n \geqslant 3$; hence the Z-property implies that $[\alpha_1\alpha_2,\beta_1] \neq 1$ or $[\alpha_1\alpha_2,\beta_2] \neq 1$. Since each $\beta$ is irreducible we may assume without loss of generality that
\begin{equation}\label{eq2}
\alpha_1\alpha_2=\beta_1 x  
\end{equation} 
for some $x \in R$.
Using \eqref{eq2} in \eqref{original} we get
\[
x \alpha_3 \cdots \alpha_n=\beta_2
\]
so $x$ is a unit, contrary to equation \eqref{eq2}.  Done with the case $m=2$, we now proceed by induction. Thus, assume that if an element has a factorization of length $\leqslant m-1$, then all factorizations of that element have the same length, and consider an equation \eqref{original} with $n>m$.  Then either (i) $[\alpha_1\alpha_2,\beta_m] \neq 1$ or (ii) $[\alpha_1\alpha_2,\beta_1\beta_2 \cdots \beta_{m-1}] \neq 1$.  In case (i) we have 
\[
\alpha_1\alpha_2=\beta_m x
\]
for some $x$, so the induction assumption implies that $\ell(x)=1$. Equation \eqref{original} becomes
\begin{equation}\label{eq3}
x \alpha_3 \cdots \alpha_n=\beta_1 \cdots \beta_{m-1}
\end{equation}
and the induction assumption applied to \eqref{eq3} yields $m-1=n-1$, i.e., $m=n$.  Case (ii) implies $\alpha_1\alpha_2=\xi x$ and $\beta_1 \cdots \beta_{m-1}=\xi y$ for some nonunit $\xi \in R$ so $\ell(x)=1$  and $\ell(y)=m-2$. The original equation \eqref{original} becomes $x\alpha_3 \cdots \alpha_n=y\beta_m$, so $\ell(y)+1=n-1$, i.e.,  $m=n$.
\end{proof}

It will also be convenient toward the end of the paper to have the following result. 

\begin{prop}\label{factor} 
\cite[Proposition 3.5]{B} Assume $R$ is atomic and integrally closed. Then the following statements are equivalent.
\begin{itemize}
\item[a)] $R$ has the Z-property
\item[b)] Whenever $f=a_0+a_1x$ and $g=b_0+b_1x$ are two primitive polynomials of degree 1 over $R$ and $c \in R$ is a constant factor of the product $fg$, then $c$ is a unit or an irreducible element of $R$. 
\end{itemize}
\end{prop}
 
In \cite[Theorem 2.4]{Za}, it was shown that if $R$ is a Krull domain with class group $Cl(R)$, then $R[x]$ is an HFD if and only if $\vert Cl(R) \vert \leqslant 2$. Also, in \cite{C}, it was shown that if $R[x]$ is an HFD, then $R$ is integrally closed and consequently, if $R$ is also Noetherian, then $R[x]$ is an HFD if and only if $R$ is a Krull domain with $\vert Cl(R) \vert \leqslant 2$. However, the half-factorial property in polynomial rings is not as well-understood in domains other than Krull and Noetherian domains.  We give one result in the next section which holds for more general domains.

\section{Polynomial Rings Over Half-Factorial Domains}

Let $R$ be an integral domain and consider the $v$-operation $I_v=(I^{-1})^{-1}$ on fractional ideals of $R$. If $I=(a_0,a_1,\ldots,a_n)_v$, then we will say that the elements $a_0,a_1,\ldots,a_n$ are \emph{$v$-generators} for the $v$-ideal $I$. We recall that (by definition) a Mori domain is one in which every $v$-ideal is $v$-generated by finitely many elements.

It is known that every $v$-ideal in a Krull domain is $v$-generated by two elements \cite[Corollary 44.6]{G}. With that in mind, we now state and prove the following theorem.

\begin{thm}
Let $R$ be a domain in which every $v$-finite $v$-ideal is $v$-generated by two elements.  Then $R[x]$ is an HFD if and only if each of the following conditions is satisfied:
\begin{itemize}
\item[(1)] $R$ is integrally closed,
\item[(2)] $R$ has the Z-property, and
\item[(3)] $(AB)^{-1}=\{ uv \mid u \in A^{-1}, v \in B^{-1} \}$ if $A,B \subseteq R$ are finitely generated ideals whose product $AB$ is primitive.
\end{itemize}
\end{thm}

\begin{proof}
(In this proof, $K$ denotes the quotient field of $R$.)

Assume each of the three given conditions is satisfied. By (1), it follows that condition a) of \cite[Theorem 3.3]{B} is satisfied. Conditions c) and d) in that theorem follow immediately from Propositions 3.3 and 3.5 of \cite{B} and our assumption that every $v$-finite $v$-ideal is $v$-generated by two elements. 

Essentially, if each of the three given conditions is satisfied, a slight modification of the proof of \cite[Theorem 3.3]{B} yields the conclusion that $R[x]$ is an HFD. As we have shown, each of the four hypotheses a)--d) in that theorem follow from (1)--(3) except b), which stated that whenever the product of two polynomials is primitive, then one of the factors in the product is superprimitive. However, the weaker condition (3) above is sufficient and can replace b). This is clear upon a detailed examination of  the proof.  

Next, we assume $R[x]$ is an HFD and show that the three conditions are satisfied. Condition (1) follows from \cite{C} and (2) follows from \cite[Proposition 3.1]{B}. For (3), assume $AB$ is primitive and some fraction $r/s \in (AB)^{-1}$ is given. We note that since $R$ is an HFD, it is also an MCD-domain \cite[22]{B}, and therefore we may assume $[r,s]=1$. Also, by assumption we may write $A_v=(a_0,a_1)_v$, $B_v=(b_0,b_1)_v$, and $(AB)_v=(c_0,c_1)_v$ for some $a_0,a_1,b_0,b_1,c_0,c_1 \in R$.  We define the polynomials $f(x)=a_1x+a_0$ and $g(x)=b_1x+b_0$, and then we consider the polynomial $h \in R[x]$ defined by the equation

\[
h=(x+r/s)fg
\]\

\noindent We show that $h$ is not irreducible. If $h$ is irreducible, we multiply both sides of the above equation by $s$ to obtain

\[
sh=(sx+r)fg
\]\

\noindent Then $s$ is not irreducible, since $R[x]$ is an HFD. This is a contradiction by \cite[Propositions 3.1 and 3.3]{B} and the fact that $r/s \in (c_0,c_1)^{-1}$.

Now, since $h$ is not irreducible, $h=FG$ for some nonunits $F,G \in R[x]$. Since $K[x]$ is a UFD, it follows that $(x+r/s)$ divides $F$ or $G$ in $K[x]$, say $F=(x+r/s)p(x)$, where $p(x) \in K[x]$. In fact, $p(x) \in R[x]$ by \cite[Theorem 10.4]{G}. Then $fg=pG$, so $p \notin R$ and $G \notin R$. Thus we may assume $p=wf$ and $G=w^{-1}g$ for some $w \in K$.  If $u=rw/s$ and $v=w^{-1}$, then $u \in A^{-1}$ and $v \in B^{-1}$. Since $r/s=uv$, condition (3) now follows.
\end{proof}

\section{Relation to Half-Factorial Domains}

In this section, we give a result which states that the Z-property and half-factorial property are equivalent under a certain set of assumptions on $R$. One of the assumptions is that $R$ satisfies the following condition $(C)$:\\

\begin{center}
$(C)$ If $I=(a_0,a_1)$ is a primitive ideal of $R$ generated by two elements, then there exists an irreducible element $\alpha \in R$ such that $\alpha \in I_v$.
\end{center}\

\begin{thm} \label{pepsi}
Let $R$ be an atomic, integrally closed domain that satisfies condition (C).  Then $R$ is an HFD if and only if $R$ has the Z-property.
\end{thm}

\begin{proof}
If $R$ has the Z-property, then $R$ is an HFD by Theorem \ref{z}. Now assume $R$ is an HFD, but $R$ does not have the Z-property. By \cite[Proposition 3.5]{B}, there exist primitive polynomials $f=a_0+a_1x$ and $g=b_0+b_1x$ such that $fg=ah$, where $h \in R[x]$ and $a$ is a nonzero nonunit that is not irreducible. Let $\alpha$ be an irreducible element such that $\alpha \in (A_f)_v$. Choose $c \in (A_f)_v$ such that $c \not\in (\alpha)$. Then $f^\prime=\alpha x+ c$ is a primitive polynomial such that $(A_{f^\prime})_v \subseteq (A_f)_v$. Similarly, we can find an irreducible element $\beta$ and a primitive polynomial $g^\prime=\beta x+d$ such that $(A_{g^\prime})_v \subseteq (A_g)_v$. Then \cite[Proposition 34.8]{G} implies $(A_{f^\prime g^\prime})_v \subseteq (A_{f g})_v \subseteq (a)$, so 

\[
f^\prime g^\prime=ah^\prime
\]\ 

\noindent for some $h^\prime \in R[x]$. Equating coefficients in this equation, we find that $a=u\alpha \beta$ for some unit $u$, since $R$ is an HFD. Dividing both sides of this equation by $\alpha \beta$, we obtain

\[
\left( x+\frac{c}{\alpha} \right) \left( x+\frac{d}{\beta} \right)=uh^\prime
\]\

\noindent Since $R$ is integrally closed, this equation and \cite[Theorem 10.4]{G} implies that $\alpha \mid c$ and $\beta \mid d$, a contradiction. 
\end{proof}

\section{An Application to the Ring $\mathbb{Q}[x,y,zx,zy]$}

The proof of the result in the previous section suggests one way (there may be others) to argue that the Krull domain 

\[
R:=\mathbb{Q}[x,y,zx,zy]
\]\

\noindent is not an HFD. Here, $\mathbb{Q}$ denotes the field of rational numbers and $x,y,z$ are indeterminates. The result that $R$ is not an HFD is at least a bit surprising because factorizations of monomials behave well, i.e., any two factorizations of a monomial have the same length. We note that the proof we now give is not valid if a field $F$ of characteristic 2 is substituted for $\mathbb{Q}$.  

\begin{prop}
The Krull domain $R:=\mathbb{Q}[x,y,zx,zy]$ is not an HFD.
\end{prop}

\begin{proof} We start by considering the two elements $f=x^2t+y^2$ and $g=x^2t+z^2x^2$ of the polynomial ring $R[t]$. The reason for considering these elements is because their product can be written

\[
\left( x^2t+y^2 \right) \left( x^2t+z^2x^2 \right) =x^2 \left( x^2t+(y^2+z^2x^2)t+z^2y^2 \right),
\]\

\noindent i.e., because $f$ and $g$ give an example for which Theorem \ref{factor} fails. Next, as suggested by Theorem \ref{pepsi}, we select irreducible elements in the ideals $(x^2,y^2)_v$ and $(x^2,z^2x^2)_v$. One such pair of irreducibles is $x^2+y^2$ and $x^2+z^2x^2$. The product of these two irreducibles can be written

\[
\left( x^2+y^2 \right) \left( x^2+z^2x^2 \right) =x^2 (x^2+y^2+z^2x^2+z^2y^2)
\]\

\noindent and it follows that $R$ is not an HFD. 
\end{proof}

\section{The Z-property in Krull Domains}

We show, among other things, that if a Krull domain $R$ has the Z-property, then $R$ is almost factorial, i.e., $R$ has torsion class group. 

For a nonzero nonunit element $x$ in a domain $R$, we will say that $x$ factors uniquely if whenever $x=\alpha_1 \alpha_2 \cdots \alpha_n=\beta_1 \beta_2 \cdots \beta_m$ and the $\alpha$'s and $\beta$'s are irreducible elements of $R$, then $m=n$ and after a suitable reordering of the indices we have $\alpha_i=u_ib_i$ for each $i$, where $u_i$ is a unit. 

\begin{prop}\label{tor}
Let $R$ be a Krull domain with nontorsion class group. Then there exists a nontorsion prime $P$ and an irreducible element $x \in P$ such that $x^2$ factors uniquely.
\end{prop}

\begin{proof}
Let $S$ denote the union of all the prime ideals which are nontorsion in $Cl(R)$. For each nonzero $x \in S$, let $\eta(x)$ be the number of distinct height-one primes which contain $x$. We note that $1 < \eta(x) <\infty$ for each $x \in S \backslash \{0\}$. Choose $s \in S$ for which $\eta=\eta(s)$ is minimal, and let $P_1, P_2, \cdots, P_\eta$ denote the height-one primes which contain $s$. We now consider the collection of all $\eta$-tuples $(f_1, f_2, \ldots, f_\eta)$ of positive integers such that $(P_1^{f_1}P_2^{f_2} \cdots P_\eta^{f_\eta})_v$ is principal. Then we consider only those $\eta$-tuples in this collection for which $f_1$ is minimal, and among those that remain, we consider only those for which $f_2$ is minimal, etc. We then select one of the minimal objects that remain after this process and denote it by $(e_1, e_2, \cdots, e_\eta)$. Now if $(x)=(P_1^{e_1}P_2^{e_1} \cdots P_\eta^{e_\eta})_v$, then the minimality assumptions imply that $x$ is irreducible. If $x^2=ab$ where $a,b$ are nonunits, then $(a)=(P_1^{f_1}P_2^{f_2} \cdots P_\eta^{f_\eta})_v$ and $(b)=(P_1^{g_1}P_2^{g_2} \cdots P_\eta^{g_\eta})_v$. The first minimality assumption implies that all the $f$'s and $g$'s are positive integers, i.e., none of them are zero. The other minimality assumptions imply $f_i=g_i=e_i$ for all $i$, so $a$ and $b$ are associates of $x$. 
\end{proof}	

\begin{thm}\label{U}
Let $R$ be a Krull domain.  If $R$ has the Z-property and $x$ is a nonunit such that $x^n$ factors uniquely for some positive integer $n >1$, then the principal ideal $(x^n)$ is a product of principal symbolic prime powers.
\end{thm}

\begin{proof}
We assume $x^n$ factors uniquely, but $(x^n)$ is not a product of principal symbolic prime powers. Note that if $a$ is a nonunit such that $a | x$, then $a^n$ factors uniquely. Therefore, we may assume that $x$ is irreducible. Then $(x)$ is not a symbolic prime power. In other words, $(x)=(P_1^{e_1}P_2^{e_2} \cdots P_n^{e_n})_v$, where $n>1$ and the $P_i$'s are height one primes with $P_i \neq P_j$ for $i \neq j$. Now choose $y \in R$ such that $P_1=(x,y)_v$, and choose $z \in {P_1}^{-1}$ such that $z \not\in R$. Let $n$ be the smallest positive integer such that $xz^n \in R$ and $xz^{n+1} \not\in R$. (Such an $n$ exists because $R$ is completely integrally closed.) Consider the equation

\[
x^2(yz)^{n+1}=y^{n+1}(x^2z^{n+1})
\]\

\noindent In this equation, we have 

\[
[x^2,y^{n+1}]=1 \ \text{and} \ [x^2,x^2z^{n+1}]=1
\]\

\noindent and therefore $R$ does not have the Z-property. Indeed, if $x | y^{n+1}$, then $y^{n+1} \in P_2$, so $y \in P_2$, a contradiction. Also, if $x |x^2z^{n+1}$, then $xz^{n+1} \in R$, a contradiction.
\end{proof}

\begin{cor}\label{torsion}
Let $R$ be a Krull domain. If $R$ has the Z-property, then $Cl(R)$ is torsion. 
\end{cor}

\begin{proof}
This follows by Proposition \ref{tor} and Theorem \ref{U}. For, assume $Cl(R)$ is not torsion. Let $P$ be a nontorsion prime and $x \in P$ be such that $x^2$ factors uniquely. Then $(x^2)$ is a product of principal symbolic prime powers. This implies $P$ is torsion, a contradiction.
\end{proof}

\begin{ex}
If $x,y,z$ are indeterminates, then the domain $F[x,y,zx,zy]$ is Krull with infinite cyclic class group \cite[Theorem 45.9]{G}. Therefore it does not have the Z-property by Corollary \ref{torsion}.
\end{ex}

\begin{ex}
Let $R$ be Dedekind with $Cl(R) \cong \mathbb{Z}_3= \{ [0], [1], [2] \}$ such that the ideal class corresponding to $[1]$ contains all nonprincipal primes.  (Such a domain exists by \cite[Corollary 1.5]{Gr}).  Then $R$ has the Z-property. 
\end{ex}

\begin{proof}
Assume $abc=de$; we claim $[ab,d]\neq 1 $ or $[ab,e] \neq 1$.  We can assume $a,b$ are nonprime irreducibles. Write $(a)=P_1P_2 P_3$ and $(b)=P_4P_5P_6$.  Then there are ideals $I,J$ such that $(d)=P_{i_1} \cdots P_{i_k} I$ and $(e)=P_{i_{k+1}} \cdots P_{i_6} J$ for some permutation $i_1, \ldots , i_6$ of the indices $1, \ldots, 6$.  If $k \geqslant 3$, then it follows $[ab,d] \neq 1$.  Else $6-k \geqslant 3$ and it follows $[ab,e] \neq 1$.  
\end{proof}

\begin{thm}
Let $R$ be an atomic integrally closed domain in which every $v$-finite $v$-ideal is $v$-generated by two elements. Then $R$ has the Z-property if and only if $R$ is an HFD and for each primitive ideal $I \subseteq R$ there exists an irreducible element $\alpha$ such that $\alpha \in I_v$.
\end{thm}

\begin{proof}
Assume $R$ has the Z-property. Then $R$ is an HFD by Theorem \ref{z}. Now suppose $I$ is a primitive ideal of $R$ and let $\alpha \in I_v$ be an element of minimal length. By assumption, we have $I=(a_0,a_1)_v$ and $I^{-1}=(b_0,b_1)_v$ where $a_0,a_1 \in R$, $b_0,b_1 \in K$, and $K$ is the quotient field of $R$. Let $f=a_1x+a_0$ and $h=b_1x+b_0$. Then $g=\alpha h \in R[x]$ and $g$ is primitive by the minimality assumption. If $\phi=fh$, then $\phi \in R[x]$. Finally, we apply Proposition \ref{factor} to the equation $fg=\alpha \phi$ to obtain that $\alpha$ is irreducible. The converse was proved in Theorem \ref{pepsi}.
\end{proof}

\begin{cor}
Let $R$ be a Krull domain. Then $R$ has the Z-property if and only if $R$ is an HFD and for each primitive ideal $I \subseteq R$ there exists an irreducible element $\alpha$ such that $\alpha \in I_v$.
\end{cor}

\bibliographystyle{alpha}

\end{document}